\documentclass{article}
\usepackage{amssymb}
\usepackage{amsmath}
\usepackage{amsfonts}

\newtheorem{theorem}{Theorem}[section]

\newtheorem{corollary}[theorem]{Corollary}

\newtheorem{definition}[theorem]{Definition}

\newtheorem{lemma}[theorem]{Lemma}

\newtheorem{proposition}[theorem]{Proposition}
\newtheorem{remark}[theorem]{Remark}

\newenvironment{proof}[1][Proof]{\noindent\textbf{#1.} }{\ \rule{0.5em}{0.5em}}
\input{tcilatex}
\begin{document}

\title{Solvability of Dirac type equations}
\author{Qingchun Ji and Ke Zhu}
\maketitle

\begin{abstract}
This paper develops a weighted $L^{2}$-method for the (half) Dirac equation.
For Dirac bundles over closed Riemann surfaces, we give a sufficient
condition for the solvability of the (half) Dirac equation in terms of a
curvature integral. Applying this to the Dolbeault-Dirac operator, we
establish an automatic transversality criteria for holomorphic curves in K%
\"{a}hler manifolds. On compact Riemannian manifolds, we give a new
perspective on some well-known results about the first eigenvalue of the
Dirac operator, and improve the estimates when the Dirac bundle has a $%
\mathbb{Z}_{2}$-grading. On Riemannian manifolds with cylindrical ends, we
obtain solvability in the $L^{2}$-space with suitable exponential weights
while allowing mild negativity of the curvature.
\end{abstract}

\tableofcontents


\section{Introduction}

In many geometric problems, it is important to determine the solvability of
the linear equation%
\begin{equation}
Du=f  \label{Dirac}
\end{equation}%
where $D$ is the Dirac operator on some Dirac bundle. For example, the
fundamental Dirac operator on spin manifolds (\cite{AS}), the
Dolbeault-Dirac operator in K\"{a}hler geometry, and the twisted Dirac
operator in the normal bundle of instantons (associative submanifolds) in $%
G_{2}$ manifolds (\cite{M}). In general, it is not easy to know when $\left( %
\ref{Dirac}\right) $ is solvable. For the Dirac operator on spin manifolds,
a sufficient condition was given by the positivity of the scalar curvature,
dating back to a theorem of Lichnerowicz. However, the positive scalar
curvature condition is not always necessary, as the Dirac operator on spin
manifolds has the remarkable conformal covariance property (\cite{Hi}), and
a conformal change of metric could make the scalar curvature negative
somewhere.

In this paper, starting with the Bochner formula, we establish weighted $%
L^{2}$-estimates and existence theorems for the Dirac equation, just as H%
\"{o}rmander's weighted $L^{2}$-method for the $\bar{\partial}$-equation (%
\cite{H1}, \cite{H2}). In applications of the $L^{2}$-method, it is very
important to construct good weight functions from geometric conditions (e.g. 
\cite{D}, \cite{S1}$\sim $\cite{S3}). Sometimes one can gain
\textquotedblleft extra positivity\textquotedblright\ in suitable weighted $%
L^{2}$-spaces to establish vanishing theorems.

Let $\lambda _{\mathbb{S}}$ be the function on $M$ defined in $\left( \ref%
{first-eigen-fcn}\right) $, which pointwisely is the first eigenvalue of
some curvature operator. For Dirac equations on $2$-dimensional Riemannian
manifolds, taking $n=2$ in Proposition \ref{Dirac-solve}, we have

\begin{theorem}
\label{RieSurfaceCase}Let $\mathbb{S}$ be a Dirac bundle over a $2$%
-dimensional Riemannian manifold $\left( M,g\right) $ and $D$ be the Dirac
operator. Suppose there exists a $C^{2}$ function $\varphi :M\rightarrow 
\mathbb{R}$ such that $\Delta \varphi +2\lambda _{\mathbb{S}}\geq 0$ on $M$.
If a section $f$ of $\mathbb{S}$ satisfies\ $\int_{M}\frac{\left\vert
f\right\vert ^{2}}{\Delta \varphi +2\lambda _{\mathbb{S}}}e^{-\varphi
}<\infty $, then there exists a section $u$ of $\mathbb{S}$ such that 
\begin{equation}
Du=f,\ \mathrm{and}\ \int_{M}\left\vert u\right\vert ^{2}e^{-\varphi }\leq
\int_{M}\frac{\left\vert f\right\vert ^{2}}{\Delta \varphi +2\lambda _{%
\mathbb{S}}}e^{-\varphi }.  \label{weight-L2-solvability}
\end{equation}
\end{theorem}

\bigskip Our theorem \ref{RieSurfaceCase} leads to the following solvability
criterion of the half Dirac equation on $\mathbb{Z}_{2}$-graded Dirac
bundles (see Section \ref{Dirac-bdl} for definitions).

\begin{corollary}
\label{curvature-integral-D+}Let $\mathbb{S}$ be a $\mathbb{Z}_{2}$-graded
Dirac bundle over a closed $2$-dimensional Riemannian manifold $M$ and $%
D^{\pm }$ be the half Dirac operators, then 
\begin{equation}
\lambda _{\mathrm{min}}(D^{\pm }D^{\mp })\geq \frac{2}{\mathrm{Vol}(M)}%
\int_{M}\lambda _{\mathbb{S}^{\mp }},  \label{eigen-est-D+-surface}
\end{equation}%
where $\lambda _{\mathrm{min}}(D^{\pm }D^{\mp })$ is the first eigenvalue of 
$D^{\pm }D^{\mp }$. Consequently, if 
\begin{equation}
\int_{M}\lambda _{\mathbb{S}^{\mp }}>0,  \label{posi-curvature-integral}
\end{equation}%
then for any $f\in L^{2}(M,\mathbb{S}^{\mp })$, there exists a section $u\in
L^{2}(M,\mathbb{S}^{\pm })$ such that $D^{\pm }u=f$.
\end{corollary}

Applying Corollary \ref{curvature-integral-D+} to the Dolbeault-Dirac
operator, we obtain

\begin{corollary}
\label{curvature-integral-dbar} Let $E$ be a holomorphic vector bundle over
a closed Riemann surface $M$ and $\overline{\partial }:E\rightarrow \wedge
^{0,1}(E)$ be the Cauchy-Riemann operator. For any Hermitian metric on $E$,
we have the following estimates: 
\begin{eqnarray}
\lambda _{\mathrm{min}}(\overline{\partial }^{\ast }\overline{\partial })
&\geq &\frac{-1}{\mathrm{Vol}(M)}\int_{M}\Theta _{E},  \label{eigen spin^c1}
\\
\lambda _{\mathrm{min}}(\overline{\partial }\overline{\partial }^{\ast })
&\geq &\frac{1}{\mathrm{Vol}(M)}\left( \int_{M}\theta _{E}+2\pi \chi
(M)\right) .  \label{eigen spin^c2}
\end{eqnarray}%
In particular, if $E$ is a line bundle, we have 
\begin{equation}
\lambda _{\mathrm{min}}(\overline{\partial }^{\ast }\overline{\partial }%
)\geq \frac{-2\pi c_{1}(E)}{\mathrm{Vol}(M)},\ \ \lambda _{\mathrm{min}}(%
\overline{\partial }\overline{\partial }^{\ast })\geq \frac{2\pi \left(
c_{1}(E)+\chi (M)\right) }{\mathrm{Vol}(M)},  \label{eigen-line-bdl}
\end{equation}%
where $\theta _{E}$ and $\Theta _{E}$ are defined by $\left( \ref{thetaE}%
\right) $ and $\left( \ref{ThetaE}\right) $ respectively, $c_{1}(E)$ is the
first Chern number of $E$, and $\chi (M)$ is the Euler number of $M$.
\end{corollary}

\bigskip The above Corollary generalizes B\"{a}r's first eigenvalue estimate
(\cite{B}) for the classic Dirac operator on closed Riemann surfaces, where $%
E=K_{M}^{\frac{1}{2}}$ and the Riemann surface is of genus $0$.

Corollary \ref{curvature-integral-dbar} has a potential extension to real
linear Cauchy-Riemann operators on Riemann surfaces, which in turn has
applications to the \textquotedblleft automatic
transversality\textquotedblright\ criteria of $J$-holomorphic curves in
almost complex manifolds $X$. We recall the definitions below.

Given a Riemann surface $\left( M,j\right) $, where $j$ is its complex
structure, a $J$\emph{-holomorphic curve} $u:\left( M,j\right) \rightarrow
\left( X,J\right) $ is a solution of the nonlinear Cauchy-Riemann equation $%
\overline{\partial }_{J}u=0$, where $\overline{\partial }_{J}\left( u\right)
=\frac{1}{2}\left( du+J\circ du\circ j\right) $. For a $J$-holomorphic curve 
$u$, the \emph{linearized operator} 
\begin{equation*}
D_{u}:W^{1,p}(M,u^{\ast }TX)\rightarrow L^{p}(M,\wedge ^{0,1}(u^{\ast }TX))%
\text{ }
\end{equation*}%
$(p\geq 2)$ of the nonlinear Cauchy-Riemann operator $\overline{\partial }%
_{J}$ is of the form%
\begin{equation*}
D_{u}=\overline{\partial }+A,
\end{equation*}%
where $\overline{\partial }$ is the Cauchy-Riemann operator on the complex
vector bundle $u^{\ast }TX$ and $A:W^{1,p}(M,u^{\ast }TX)\rightarrow
L^{p}(M,\wedge ^{0,1}(u^{\ast }TX))$ \ is an anti-complex linear
homomorphism determined by the Nijenhuis tensor of $J$. $A=0$ if $J$ is
integrable. $\overline{\partial }+A$ is a \emph{real linear Cauchy-Riemann
operator} (c.f. Appendix C \cite{MS}).

\begin{definition}
\label{Fred-reg}A $J$-holomorphic curve $u$ is called \emph{transversal} (or 
\emph{Fredholm regular}), if $D_{u}$ is surjective. A $J$-holomorphic curve $%
u$ has automatic transversality, if $u$ is transversal regardless of whether 
$J$ is generic or not.
\end{definition}

Note for symplectic manifolds $\left( X,\omega \right) $ of dimension $4$
with compatible almost complex structures $J$, for immersed $J$-holomorphic
curves $u$, taking $E=N_{u}$ (the normal bundle of $u$ in $X$), the estimate 
$\left( \ref{eigen-line-bdl}\right) $ is related to the well-known Chern
number condition $c_{1}(E)>-\chi (M)$ to ensure automatic transversality,
which was first mentioned in \cite{Gr} and later established in \cite{HLS}, 
\cite{W}, \cite{MS}, with nice applications to the moduli space theory of $J$%
-holomorphic curves.\ 

If the almost complex structure $J$ is integrable, our Corollary \ref%
{curvature-integral-dbar} provides a criterion for automatic transversality
of holomorphic curves in K\"{a}hler manifolds with real dimensional $\geq 4$.

\begin{corollary}
\label{auto-trans} Let $u:\left( M,j\right) \rightarrow \left( X,\omega
,J\right) $ be a holomorphic curve in a K\"{a}hler manifold $X$ equipped
with complex structure $J$ and K\"{a}hler form $\omega $. Let $E=u^{\ast }TX$
be the vector bundle over $M$ with induced holomorphic structure and metric.
Then $u$ has automatic transversality, if 
\begin{equation}
\int_{M}\theta _{E}>-2\pi \chi (M),  \label{curv-intg-trans}
\end{equation}%
where $\theta _{E}$ is defined in $\left( \ref{thetaE}\right) $.
\end{corollary}

We leave the general case to our future work \cite{JZ}.

\medskip

For spinor bundles over closed spin manifolds of dimension $\geq 3$, Hijazi
obtained (\cite{H}) a lower bound of the spectrum of the fundamental Dirac
operator in terms of the first eigenvalue of the Yamabe operator, with the
help of the conformal covariance of the Dirac operator and the
transformation law of the scalar curvature under a conformal change of the
given metric. By introducing a new connection, B\"{a}r (\cite{B})
generalized this result, using a delicate technique of completing square, to
any Dirac bundle over a closed Riemannian manifold.

As an application of our weighted $L^{2}$-estimates, we construct weight
functions by solving certain partial differential equations and give a new
proof of B\"{a}r's results on the first eigenvalue of the Dirac operator
(Theorem \ref{first-eigen-est}).

For Dirac bundles with $\mathbb{Z}_{2}$-gradings, we establish similar
results for the (half) Dirac operator $D^{\pm }$ in Section \ref{Z2-Dirac},
as this is important for many geometric applications. Using the $\mathbb{Z}%
_{2}$-grading (which always exists on Dirac bundles over even dimensional
manifolds) we improve B\"{a}r's (\cite{B} Theorem 3) first eigenvalue
estimates of $D$ on even dimensional manifolds. \bigskip \medskip

In gauge theory, it occurs often that $M$ are Riemannian manifolds with 
\emph{cylindrical ends} (Definition \ref{cylindrical}), and exponential
weights on the ends are frequently used to set up the moduli spaces (e.g. 
\cite{Don}, \cite{T}). On such manifolds we have the following existence
theorem without assuming $\lambda _{\mathbb{S}}$ is positive everywhere.

\begin{theorem}
\label{cylindrical-solvability}Suppose $\mathbb{S}$ is a Dirac bundle over a
Riemannian manifold $\left( M,g\right) $ with cylindrical ends, $\dim M\geq
3 $. Suppose for some compact subset $K$, $M\backslash K$ is contained in
the cylindrical ends, and there exist a constant $\alpha >0$ such that 
\begin{equation}
\lambda _{\mathbb{S}}\geq \alpha \text{ on }M\backslash K.
\label{ls-condition-cyl-mfd}
\end{equation}%
Then there exists a constant $\beta >0$ with the following significance:
when 
\begin{equation}
\lambda _{\mathbb{S}}\geq -\beta \text{ on }K,
\label{mild-negative-curvature}
\end{equation}%
we have

\begin{enumerate}
\item for any section $f\in L_{\delta }^{2}(M,\mathbb{S})$, there exists a
section $u\in W_{\delta }^{1,2}(M,\mathbb{S})$ such that 
\begin{equation}
Du=f\ \mathrm{and}\ \left\Vert u\right\Vert _{W_{\delta }^{1,2}(M,\mathbb{S}%
)}\leq C\left\Vert f\right\Vert _{L_{\delta }^{2}(M,\mathbb{S})};
\label{weighted-sol}
\end{equation}

\item for any section $f\in L^{p}\left( M,\mathbb{S}\right) $ ($p>1$), there
exists a section $u\in W^{1,p}\left( M,\mathbb{S}\right) $ such that 
\begin{equation}
Du=f\ \mathrm{and}\ \left\Vert u\right\Vert _{W^{1,p}\left( M,\mathbb{S}%
\right) }\leq C_{p}\left\Vert f\right\Vert _{L^{p}\left( M,\mathbb{S}\right)
}.  \label{weighted-Lp-cyl}
\end{equation}
\end{enumerate}

Here $\delta _{1},\cdots ,\delta _{m}\geq 0$ are sufficiently small, $\delta
=\left( \delta _{1},\cdots ,\delta _{m}\right) $ is the exponential weight
(defined in $\left( \ref{exp-weight}\right) $) for the cylindrical ends, and
the constants $C$ and $C_{p}$ are independent on $f$.
\end{theorem}

Note that the $L^{2}$-solvability of the Dirac equation on a complete
Riemannian manifold was established in Theorem 2.11 of \cite{GL}, with the
assumption that $\lambda _{\mathbb{S}}\geq \alpha >0$. Our theorem allows
mild negativity of $\lambda _{\mathbb{S}}$ on a compact domain $K\subset M$
(as in $\left( \ref{mild-negative-curvature}\right) $), when $M$ is a
Riemannian manifold with cylindrical ends. Some recent results of the Dirac
equation on non-compact Riemannian manifolds can be found in \cite{B1}, and 
\cite{Gr06}$\sim $\cite{Gr12}.\bigskip\ 

\bigskip \textbf{Acknowledgement. }We would like to thank Professor Yum-Tong
Siu and Professor Clifford Taubes for interest and support. We thank Harvard
math department for the excellent research environment, where the majority
of the work was carried out. Discussions with Hai Lin, Baosen Wu, Yi Xie and
Jie Zhou were helpful. Last, we thank the anonymous referee for pointing out
many useful references. The work of Q. Ji is partially supported by
NSFC11171069 and NSFC11322103.

\section{Weighted $L^2$-estimates for Dirac operators}

\subsection{Dirac bundles and Dirac operators\label{Dirac-bdl}}

In this section, we recall some basic facts of the Dirac operator and set up
the notations . Let $(M,g)$ be a smooth $n$-dimensional Riemannian manifold
and $C\ell \left( M,g\right) \rightarrow M$ be the corresponding \emph{%
Clifford bundle}. Let $\mathbb{S}\rightarrow M$ be a bundle of left $C\ell
\left( M,g\right) $-modules endowed with a Riemannian metric $\langle \cdot
,\cdot \rangle $ and a Riemannian connection $\nabla $ such that at any $%
x\in M$, for any unit vector $e\in T_{x}M$ and any $s,s^{\prime }\in \mathbb{%
S}_{x}$, 
\begin{equation}
\left\langle e\cdot s,e\cdot s^{\prime }\right\rangle =\left\langle
s,s^{\prime }\right\rangle  \label{orthogonal}
\end{equation}%
where the $\cdot $ is the Clifford multiplication. Furthermore, for any
smooth vector field $V$ on $M$ and smooth section $s$ of $\mathbb{S}$, 
\begin{equation}
\nabla \left( V\cdot s\right) =\left( \nabla V\right) \cdot s+V\cdot \left(
\nabla s\right) ,  \label{Leibniz}
\end{equation}%
where on the right hand side, $\nabla $ are the covariant derivatives of the
Levi-Civita connection on $M$ and a connection on $\mathbb{S}$ for the first
and second terms respectively.

\begin{definition}
A bundle $(\mathbb{S},\langle\cdot,\cdot\rangle,\nabla)$ of $C\ell\left(
M,g\right) $-modules satisfying $\left( \ref{orthogonal}\right) ,\left( \ref%
{Leibniz}\right) $ is called a \emph{Dirac bundle} over $M$ (Definition 5.2 
\cite{LM}) and has a canonically associated \emph{Dirac operator} $D$ such
that for any section $s$ of $\mathbb{S}$, 
\begin{equation}
Ds=\sum_{i=1}^{n}e_{i}\cdot \nabla _{e_{i}}s,  \label{DiracOper}
\end{equation}%
where $\left\{ e_{i}\right\} _{i=1}^{n}$ is any orthonormal basis of $T_{x}M$
for $x$ on $M$.
\end{definition}

On the Dirac bundle $\mathbb{S}$ we can define the \emph{canonical section} $%
\mathfrak{R}$ of $\mathrm{End}\left( \mathbb{S}\right) $, such that for any
smooth section $s$ of $\mathbb{S}$, 
\begin{equation}
\mathfrak{R}\left( s\right) =\frac{1}{2}\sum_{i,j=1}^{n}e_{i}\cdot
e_{j}\cdot R_{e_{i},e_{j}}\left( s\right) ,  \label{R-operator}
\end{equation}%
where $R_{V,W}$ is the curvature transform on $\mathbb{S}$. Then we have the
Bochner formula (c.f. Theorem 8.2 in Chapter II \cite{LM}) 
\begin{equation}
D^{2}=\nabla ^{\ast }\nabla +\mathfrak{R}.  \label{Gen-Bochner}
\end{equation}%
Especially, when $\mathbb{S}$ is the \emph{spinor bundle} over a spin
manifold $M\,$, Lichnerowicz's theorem says 
\begin{equation}
\mathfrak{R}=\frac{1}{4}R\cdot \mathrm{Id}_{\mathbb{S}},  \label{spin-R}
\end{equation}%
where $R$ is the scalar curvature of $\left( M,g\right) $, and $\mathrm{Id}_{%
\mathbb{S}}$ is the identity map on $\mathbb{S}$.

For later applications, for any $x$ on $M$ and $\mathfrak{R}\left( x\right)
\in \mathrm{End}\left( \mathbb{S}_{x}\right) $, we let the function 
\begin{equation}
\lambda _{\mathbb{S}}\left( x\right) =\text{the smallest eigenvalue of }%
\mathfrak{R}\left( x\right) ,  \label{first-eigen-fcn}
\end{equation}%
where $\mathfrak{R}\left( x\right) $ is defined in $\left( \ref{R-operator}%
\right) $. Since $\mathfrak{R}\left( x\right) $ is differentiable with
respect to $x$ on $M$, using the definition of eigenvalues by the Rayleigh
quotient, it is not hard to see $\lambda _{\mathbb{S}}\left( x\right) $ is a 
\emph{Lipschitz} function on $M$.

\bigskip

\bigskip In several important applications (see Section \ref{Z2-Dirac}), one
needs to consider the $\mathbb{Z}_{2}$\emph{-graded Dirac bundle}, i.e. $%
\mathbb{S}$ has a parallel decomposition $\mathbb{S=S}^{+}\mathbb{\oplus S}%
^{-}$ so that $Cl^{i}\left( M\right) \cdot \mathbb{S}^{j}\subseteq \mathbb{S}%
^{ij}$ for all $i,j\in \mathbb{Z}_{2}\simeq \left\{ +,-\right\} $, where $ij$
is the sign of the product of the signs $i$ and $j$, and $Cl^{i}\left(
M\right) $ ($i=+,-$) are the even and odd parts of $Cl\left( M\right) $. The
Dirac operator $D$ splits accordingly%
\begin{equation}
D=\left[ 
\begin{array}{cc}
0 & D^{-} \\ 
D^{+} & 0%
\end{array}%
\right] ,  \label{pos-neg-Dirac}
\end{equation}%
where $D^{\pm }:\Gamma \left( \mathbb{S}^{\pm }\right) \rightarrow \Gamma
\left( \mathbb{S}^{\mp }\right) $, and $D^{+}$ and $D^{-}$ are adjoint of
each other. Note that for Dirac bundles over \emph{even} dimensional
Riemannian manifolds, there always exists a natural $\mathbb{Z}_{2}$-grading
(c.f. Section 6 \cite{LM}). By $\left( \ref{R-operator}\right) $, we see the
curvature operator $\mathfrak{R}\left( \cdot \right) $ splits as well: 
\begin{equation}
\mathfrak{R}\left( \cdot \right) =\left[ 
\begin{array}{cc}
\mathfrak{R}^{+}\left( \cdot \right) & 0 \\ 
0 & \mathfrak{R}^{-}\left( \cdot \right)%
\end{array}%
\right] ,  \label{curvature-pm}
\end{equation}%
where $\mathfrak{R}^{+}\left( \cdot \right) \in \mathrm{End}\left( \mathbb{S}%
^{+},\right) $,$\mathbb{\ }$and $\mathfrak{R}^{-}\left( \cdot \right) \in 
\mathrm{End}\left( \mathbb{S}^{-}\right) $.

\bigskip Let $(M,g)$ be a Riemann surface endowed with a metric and $(E,h)$
be a Hermitian vector bundle over $M.$ We denote the twisted $\mathrm{spin}%
^{c}$ bundle by 
\begin{equation*}
\mathbb{S}=\wedge ^{0,\ast }(M)\otimes E
\end{equation*}%
whose Dirac operator is given by 
\begin{equation*}
D=\sqrt{2}(\overline{\partial }+\overline{\partial }^{\ast }).
\end{equation*}%
There is a natural $\mathbb{Z}_{2}$-graded structure given by $\mathbb{S}=%
\mathbb{S}^{+}\oplus \mathbb{S}^{-}$ where 
\begin{equation*}
\mathbb{S}^{+}=E,\ \ \ \mathbb{S}^{-}=\wedge ^{0,1}(E).
\end{equation*}%
It is easy to see that 
\begin{equation*}
D^{+}=\sqrt{2}\overline{\partial },\ \ \ D^{-}=\sqrt{2}\overline{\partial }%
^{\ast }.
\end{equation*}%
We introduce the following functions 
\begin{eqnarray}
\theta _{E}\left( x\right) &=&\mathrm{the\ \ smallest\ \ eigenvalue\ \ of}\
\ \sqrt{-1}\Lambda R^{E}\left( x\right) ,  \label{thetaE} \\
\Theta _{E}\left( x\right) &=&\mathrm{the\ \ biggest\ \ eigenvalue\ \ of}\ \ 
\sqrt{-1}\Lambda R^{E}\left( x\right) ,  \label{ThetaE}
\end{eqnarray}%
for $x$ on $M$, where $R^{E}$ is the curvature of the Chern connection $%
\nabla $ of $(E,h)$, and $\Lambda $ is the dual Lefschetz operator of $(M,g)$%
. In this case, we can express $\mathfrak{R}^{\pm }$ as follows.

\begin{lemma}
\label{curve. term Spinc} 
\begin{equation}
\mathfrak{R}^{-} =\sqrt{-1}\Lambda R^{E} + K, \ \ \ \ \mathfrak{R}^{+} =-%
\sqrt{-1}\Lambda R^{E},  \label{R}
\end{equation}%
where $K$ is the Gaussian curvature of $(M,g)$.
\end{lemma}

\begin{proof}
We can verify (\ref{R}) by a tedious computation according to the definition
of $\mathfrak{R}^{\pm }$. We give an alternative proof by using the $%
\overline{\partial }$-Bochner formula. We only prove the first one, the
second formula can be proved in the same way. From the $\overline{\partial }$%
-Bochner formula, it follows that on $\mathbb{S}^{-}$ we have 
\begin{equation*}
D^{2}=-2\nabla ^{\bar{1}}\nabla _{\bar{1}}+2\sqrt{-1}\Lambda R^{E}+2K.
\end{equation*}%
Here we have used local coordinate such that $g_{1\bar{1}}=1$ at a given
point. Since 
\begin{equation*}
\nabla ^{\ast }\nabla =-\nabla ^{\bar{1}}\nabla _{\bar{1}}-\nabla ^{1}\nabla
_{1}=-2\nabla ^{\bar{1}}\nabla _{\bar{1}}+[\nabla _{1},\nabla _{\bar{1}}],
\end{equation*}%
we get 
\begin{equation*}
D^{2}=\nabla ^{\ast }\nabla +\sqrt{-1}\Lambda R^{E}+K
\end{equation*}%
which, comparing with the Bochner formula (\ref{Gen-Bochner}) for Dirac
operators, implies the first formula.
\end{proof}

\subsection{Weighted $L^{2}$-estimates and existence results\label{proofs}}

Let $\mathbb{S}$ be a Dirac bundle over a smooth Riemannian manifold $\left(
M,g\right) $ and $D$ be the Dirac operator. For any smooth sections $s\in
\Gamma (M,\mathbb{S})$ with compact support, by $\left( \ref{Gen-Bochner}%
\right) $ and integration by parts, we have 
\begin{equation}
\int_{M}\left\vert Ds\right\vert ^{2}=\int_{M}\left( \left\vert \nabla
s\right\vert ^{2}+\left\langle s,\mathfrak{R}s\right\rangle \right) .
\label{L2-identity}
\end{equation}

\begin{definition}
(Weighted $L^{2}$-space)\label{weighted-L2-space} Let $\varphi :M\rightarrow 
\mathbb{R}$ be a $C^{2}$ function. For any sections $s$ and $s^{\prime }$ of 
$\mathbb{S}$, let the \emph{weighted }inner product of $s$ and $s^{\prime }$
be%
\begin{equation*}
\left( s,s^{\prime }\right) _{\varphi }=\int_{M}\left\langle s,s^{\prime
}\right\rangle e^{-\varphi }.
\end{equation*}
where $\varphi :M\rightarrow \mathbb{R}$ is a $C^2$ function. Let $%
\left\Vert s\right\Vert _{\varphi }=\sqrt{\left( s,s\right) _{\varphi }} $
and denote by $L_{\varphi }^{2}\left( M,\mathbb{S}\right) $ be the space of
sections $s$ of $\mathbb{S}$ such that $\left\Vert s\right\Vert _{\varphi
}<\infty $. We will drop the subscript $\varphi $ when $\varphi =0$.
\end{definition}

For the Dirac operator $D:L_{\varphi }^{2}\left( M,\mathbb{S}\right)
\rightarrow L_{\varphi }^{2}\left( M,\mathbb{S}\right) $, let $D_{\varphi
}^{\ast }$ be its formal adjoint with respect to the measure $e^{-\varphi
}dvol_{g}$. For $D_{\varphi }^{\ast }$, we have the following identity which
is immediate from definitions.

\begin{lemma}
\label{weight-relation}For any smooth section $s$ of $\mathbb{S}$, we have 
\begin{equation}
D_{\varphi }^{\ast }s=e^{\varphi }D\left( e^{-\varphi }s\right) =-\nabla
\varphi \cdot s+Ds.  \label{weighted-Dirac-relation}
\end{equation}
\end{lemma}

\bigskip We will derive a weighted version of $\left( \ref{L2-identity}%
\right) $ by the technique in \cite{JLY}. Let $\Delta $ be the
Laplace-Beltrami operator on $\left( M,g\right) $.

\begin{proposition}
\label{weighted-L2-prop}For any smooth section $s$ of $\mathbb{S}$ with
compact support and any $C^{2}$ function $\varphi :M\rightarrow \mathbb{R}$,
we have%
\begin{eqnarray}
&&\frac{n-1}{n}\int_{M}\left\vert D_{\varphi }^{\ast }s\right\vert
^{2}e^{-\varphi }+\frac{n-2}{n}\mathrm{Re}\int_{M}\left\langle \nabla
\varphi \cdot s,D_{\varphi }^{\ast }s\right\rangle e^{-\varphi }  \notag \\
&\geq &\int_{M}\left[ \frac{1}{2}\Delta \varphi -\left( \frac{1}{2}-\frac{1}{%
n}\right) \left\vert \nabla \varphi \right\vert ^{2}+\lambda _{\mathbb{S}}%
\right] \left\vert s\right\vert ^{2}e^{-\varphi }.  \label{weighted-L2-est1}
\end{eqnarray}%
Let $\varepsilon >0$ be a constant, we have 
\begin{equation}
\int_{M}\left\vert D_{\varphi }^{\ast }s\right\vert ^{2}e^{-\varphi }\geq
C\int_{M}\left[ \Delta \varphi -\left( 1-\frac{2}{n}\right) \left( 1+\frac{1%
}{\varepsilon }\right) \left\vert \nabla \varphi \right\vert ^{2}+2\lambda _{%
\mathbb{S}}\right] \left\vert s\right\vert ^{2}e^{-\varphi }.
\label{weighted-L2-est}
\end{equation}%
where $C=\frac{n}{2\left( n-1\right) +\left( n-2\right) \varepsilon }.$
Especially, when $n=2$ we have%
\begin{equation}
\int_{M}\left\vert D_{\varphi }^{\ast }s\right\vert ^{2}e^{-\varphi }\geq
\int_{M}\left( \Delta \varphi +2\lambda _{\mathbb{S}}\right) \left\vert
s\right\vert ^{2}e^{-\varphi }.  \label{weighted-L2-est2}
\end{equation}
\end{proposition}

\begin{proof}
Set $\sigma :=e^{-\frac{\varphi }{2}}s$, then we know by $\left( \ref%
{weighted-Dirac-relation}\right) $ 
\begin{eqnarray}
\int_{M}\left\vert D_{\varphi }^{\ast }s\right\vert ^{2}e^{-\varphi }
&=&\int_{M}\left\vert e^{\varphi }D\left( e^{-\varphi }s\right) \right\vert
^{2}e^{-\varphi }  \notag  \label{one} \\
&=&\int_{M}\left\vert D\sigma -\frac{1}{2}\nabla \varphi \cdot \sigma
\right\vert ^{2}  \notag \\
&=&\int_{M}\left( \left\vert D\sigma \right\vert ^{2}+\frac{1}{4}\left\vert
\nabla \varphi \cdot \sigma \right\vert ^{2}-\mathrm{Re}\left\langle \nabla
\varphi \cdot \sigma ,D\sigma \right\rangle \right)  \notag \\
&=&\int_{M}\left( \left\vert \nabla \sigma \right\vert ^{2}+\left\langle
\sigma ,\mathfrak{R}\sigma \right\rangle +\frac{1}{4}\left\vert \nabla
\varphi \right\vert ^{2}\left\vert \sigma \right\vert ^{2}-\mathrm{Re}
\left\langle \nabla \varphi \cdot \sigma ,D\sigma \right\rangle \right) , 
\notag \\
&&
\end{eqnarray}%
where in the last identity we have used $\left( \ref{L2-identity}\right) $.

We first rewrite $\int_{M}\left\vert \nabla \sigma \right\vert ^{2}$ as
follows 
\begin{eqnarray}
\int_{M}\left\vert \nabla \sigma \right\vert ^{2} &=&\int_{M}\left\vert
\nabla s-\frac{1}{2}d\varphi \otimes s\right\vert ^{2}e^{-\varphi }  \notag
\\
&=&\int_{M}\left( \left\vert \nabla s\right\vert ^{2}+\frac{1}{4}\left\vert
\nabla \varphi \right\vert ^{2}\left\vert s\right\vert ^{2}\right)
e^{-\varphi }-\mathrm{Re}\int_{M}\left\langle \nabla s,d\varphi \otimes
s\right\rangle e^{-\varphi }.  \label{two}
\end{eqnarray}

For the last term, we have%
\begin{eqnarray}
-\mathrm{Re}\int_{M}\left\langle \nabla s,d\varphi \otimes s\right\rangle
e^{-\varphi } &=&-\mathrm{Re}\sum_{i=1}^{n}\int_{M}\left\langle \nabla
_{e_{i}}s,d\varphi \left( e_{i}\right) s\right\rangle e^{-\varphi }  \notag
\\
&=&\frac{1}{2}\mathrm{Re}\int_{M}\nabla _{\nabla \left( e^{-\varphi }\right)
}\left\vert s\right\vert ^{2}  \notag \\
&=&\frac{1}{2}\int_{M}\left[ \mathrm{div}\left( \left\vert s\right\vert
^{2}\nabla \left( e^{-\varphi }\right) \right) -\left\vert s\right\vert
^{2}\Delta \left( e^{-\varphi }\right) \right]  \notag \\
&=&\frac{1}{2}\int_{M}\left( \Delta \varphi -\left\vert \nabla \varphi
\right\vert ^{2}\right) \left\vert s\right\vert ^{2}e^{-\varphi },
\label{three}
\end{eqnarray}%
where in the third line we have used $Vf=\mathrm{div}\left( fV\right) -f%
\mathrm{div}V$ for the vector field $V=\nabla \left( e^{-\varphi }\right) $
and function $f=\left\vert s\right\vert ^{2}$.

From $\left( \ref{two}\right) $ and $\left( \ref{three}\right),$ we obtain%
\begin{eqnarray}
&&\int_{M}\left\vert \nabla \sigma \right\vert ^{2}  \notag \\
&=&\int_{M}\left( \left\vert \nabla s\right\vert ^{2}+\frac{1}{4}\left\vert
\nabla \varphi \right\vert ^{2}\left\vert s\right\vert ^{2}\right)
e^{-\varphi }+\frac{1}{2}\int_{M}\left( \Delta \varphi -\left\vert \nabla
\varphi \right\vert ^{2}\right) \left\vert s\right\vert ^{2}e^{-\varphi } 
\notag \\
&=&\int_{M}\left[ \left\vert \nabla s\right\vert ^{2}+\left( \frac{1}{2}%
\Delta \varphi -\frac{1}{4}\left\vert \nabla \varphi \right\vert ^{2}\right)
\left\vert s\right\vert ^{2}\right] e^{-\varphi }.  \label{four}
\end{eqnarray}

For the term $\mathrm{Re}\left\langle \nabla \varphi \cdot \sigma ,D\sigma
\right\rangle $, we have%
\begin{eqnarray}
\mathrm{Re}\left\langle \nabla \varphi \cdot \sigma ,D\sigma \right\rangle
&=&\mathrm{Re}\left\langle e^{-\frac{\varphi }{2}}\nabla \varphi \cdot
s,D\left( e^{-\frac{\varphi }{2}}s\right) \right\rangle  \notag \\
&=&\mathrm{Re}\left\langle \nabla \varphi \cdot s,e^{\frac{\varphi }{2}%
}D\left( e^{-\frac{\varphi }{2}}s\right) \right\rangle e^{-\varphi }  \notag
\\
&=&\mathrm{Re}\left\langle \nabla \varphi \cdot s,D_{\frac{\varphi }{2}%
}^{\ast }\left( s\right) \right\rangle e^{-\varphi }  \notag \\
&=&\mathrm{Re}\left\langle \nabla \varphi \cdot s,D_{\varphi }^{\ast }\left(
s\right) +\frac{1}{2}\nabla \varphi \cdot s\right\rangle e^{-\varphi },
\label{Five}
\end{eqnarray}%
where in the fourth identity we have used Lemma \ref{weight-relation}.

By the Cauchy-Schwarz inequality, we get 
\begin{equation}
\left\vert \nabla s\right\vert ^{2}\geq \frac{1}{n}\left\vert Ds\right\vert
^{2}.  \label{grad-and-Ds}
\end{equation}

Combining $\left( \ref{one}\right) $, $\left( \ref{four}\right) $, $\left( %
\ref{Five}\right) $, $\left( \ref{grad-and-Ds}\right) $ and using $\left( %
\ref{weighted-Dirac-relation}\right) $, we have%
\begin{eqnarray}
&&\int_{M}\left\vert D_{\varphi }^{\ast }s\right\vert ^{2}e^{-\varphi } 
\notag \\
&\geq &\int_{M}\left[ \frac{1}{n}\left\vert D_{\varphi }^{\ast }s+\nabla
\varphi \cdot s\right\vert ^{2}+\left( \frac{1}{2}\Delta \varphi -\frac{1}{4}%
\left\vert \nabla \varphi \right\vert ^{2}\right) \left\vert s\right\vert
^{2}\right] e^{-\varphi }  \notag \\
&&+\int_{M}\left( \lambda _{\mathbb{S}}\left\vert s\right\vert ^{2}+\frac{1}{%
4}\left\vert \nabla \varphi \right\vert ^{2}\left\vert s\right\vert
^{2}\right) e^{-\varphi }-\mathrm{Re}\int_{M}\left\langle \nabla \varphi
\cdot s,D_{\varphi }^{\ast }s+\frac{1}{2}\nabla \varphi \cdot s\right\rangle
e^{-\varphi }  \notag \\
&=&\int_{M}\left[ \frac{1}{2}\Delta \varphi -\left( \frac{1}{2}-\frac{1}{n}%
\right) \left\vert \nabla \varphi \right\vert ^{2}+\lambda _{\mathbb{S}}%
\right] \left\vert s\right\vert ^{2}e^{-\varphi }  \notag \\
&&-2\left( \frac{1}{2}-\frac{1}{n}\right) \mathrm{Re}\int_{M}\left\langle
\nabla \varphi \cdot s,D_{\varphi }^{\ast }s\right\rangle e^{-\varphi }+%
\frac{1}{n}\int_{M}\left\vert D_{\varphi }^{\ast }s\right\vert
^{2}e^{-\varphi }  \label{id} \\
&\geq &\int_{M}\left[ \frac{1}{2}\Delta \varphi -\left( \frac{1}{2}-\frac{1}{%
n}\right) \left( 1+\frac{1}{\varepsilon }\right) \left\vert \nabla \varphi
\right\vert ^{2}+\lambda _{\mathbb{S}}\right] \left\vert s\right\vert
^{2}e^{-\varphi }  \notag \\
&&-\left( \left( \frac{1}{2}-\frac{1}{n}\right) \varepsilon -\frac{1}{n}%
\right) \int_{M}\left\vert D_{\varphi }^{\ast }s\right\vert ^{2}e^{-\varphi
},  \notag
\end{eqnarray}
where $\varepsilon $ is any positive constant. We have thus proved (\ref%
{weighted-L2-est}). The inequality (\ref{id}) also gives the estimate (\ref%
{weighted-L2-est1}).
\end{proof}

\bigskip

To establish the $L^{2}$-existence theorem, we also need the following
variant of Riesz representation Theorem:

\begin{lemma}
\label{Lemma-Hormander}(c.f. \cite{H1}) Let $T:H_{1}\longrightarrow H_{2}$
be a closed and densely defined operator between Hilbert spaces. For any $%
f\in H_{2}$ and any constant $C>0$, the following conditions are equivalent

\begin{enumerate}
\item there exists some $u\in \mathrm{Dom}(T)$ such that $Tu=f$ and $\Vert
u\Vert _{H_{1}}\leq C$.

\item $|(f,s)_{H_{2}}|\leq C\Vert T^{\ast }s\Vert _{H_{1}}$ holds for any $%
s\in \mathrm{Dom}(T^{\ast })$.
\end{enumerate}
\end{lemma}

We have the following proposition of the $L^{2}$-existence result for the
Dirac operator.

\begin{proposition}
\label{Dirac-solve}Let $\mathbb{S}$ be a Dirac bundle over a smooth
Riemannian manifold $\left( M,g\right) $ and $D$ be the Dirac operator. Let $%
\varphi :M\rightarrow \mathbb{R}$ be a $C^{2}$ function and $\varepsilon >0$
be a constant such that 
\begin{equation}
\Delta \varphi -\left( 1-\frac{2}{n}\right) \left( 1+\frac{1}{\varepsilon }%
\right) \left\vert \nabla \varphi \right\vert ^{2}+2\lambda _{\mathbb{S}%
}\geq 0\text{ on }M.  \label{positivity}
\end{equation}%
For any section $f\in L_{\varphi }^{2}(M,\mathbb{S})$, if 
\begin{equation}
\int_{M}\frac{\left\vert f\right\vert ^{2}}{\Delta \varphi -\left( 1-\frac{2%
}{n}\right) \left( 1+\frac{1}{\varepsilon }\right) \left\vert \nabla \varphi
\right\vert ^{2}+2\lambda _{\mathbb{S}}}e^{-\varphi }<\infty ,
\label{f-conditon}
\end{equation}%
then there exists a section $u\in L_{\varphi }^{2}(M,\mathbb{S})$ such that 
\begin{equation*}
Du=f\ \mathrm{and}\ \left\Vert u\right\Vert _{\varphi }^{2}\leq \frac{1}{C}%
\int_{M}\frac{\left\vert f\right\vert ^{2}}{\Delta \varphi -\left( 1-\frac{2%
}{n}\right) \left( 1+\frac{1}{\varepsilon }\right) \left\vert \nabla \varphi
\right\vert ^{2}+2\lambda _{\mathbb{S}}}e^{-\varphi },
\end{equation*}%
where $C=\frac{n}{2\left( n-1\right) +\left( n-2\right) \varepsilon }$.
\end{proposition}

\begin{proof}
Fix some open subset $\Omega $ which is relatively compact in $M$, we
restrict $\mathbb{S}$ to $\Omega $ and introduce 
\begin{equation*}
H_{1}=H_{2}=L_{\varphi }^{2}(\Omega ,\mathbb{S}).
\end{equation*}%
We let the closed, densely defined operator%
\begin{equation*}
T:\mathrm{Dom}(T)=\{u\in H_{1}:Du\in H_{1}\}\longrightarrow H_{2}
\end{equation*}%
to be the maximal differential operator defined by the Dirac operator(acting
on smooth sections) $\ D:\Gamma (\Omega ,\mathbb{S})\rightarrow \Gamma
(\Omega ,\mathbb{S})$.

From Proposition \ref{weighted-L2-prop} we have for any compact supported
section $s\in \Gamma (\Omega ,\mathbb{S})$, 
\begin{equation}
\int_{M}\left\vert D_{\varphi }^{\ast }s\right\vert ^{2}e^{-\varphi }\geq
C\int_{M}\left( \Delta \varphi -\left( 1-\frac{2}{n}\right) \left( 1+\frac{1%
}{\varepsilon }\right) \left\vert \nabla \varphi \right\vert ^{2}+2\lambda _{%
\mathbb{S}}\right) \left\vert s\right\vert ^{2}e^{-\varphi }.
\label{L2-estimate}
\end{equation}%
Since $T^{\ast }$ is given by the minimal differential operator defined by
the formally adjoint operator$\ D_{\varphi }^{\ast }:\Gamma (\Omega ,\mathbb{%
S})\rightarrow \Gamma (\Omega ,\mathbb{S})$(c.f. \cite{H2} for a more
general result), we know that smooth sections of $\mathbb{S}$ with compact
support in $\Omega $ are dense in $\mathrm{Dom}(T^{\ast })$ with respect to
the graph norm: 
\begin{equation*}
s\mapsto \sqrt{\Vert s\Vert _{H_{2}}^{2}+\Vert T^{\ast }s\Vert _{H_{1}}^{2}}.
\end{equation*}%
Because $\varphi $ and $\nabla \varphi $ are bounded on $\Omega $, $\left( %
\ref{L2-estimate}\right) $ holds for any $s\in \mathrm{Dom}(T^{\ast })$.

Let $f_{\Omega }\in L_{\varphi }^{2}(\Omega ,\mathbb{S})$. As we have proved
that Proposition \ref{weighted-L2-prop} is true for any $s\in \mathrm{Dom}%
(T^{\ast })$, by Cauchy-Schwarz inequality, we get the following estimate
for any $s\in \mathrm{Dom}(T^{\ast })$ 
\begin{equation*}
\left\vert (f_{\Omega },s)_{H_{2}}\right\vert ^{2}\leq \frac{1}{C}%
\int_{\Omega }\frac{\left\vert f_{\Omega }\right\vert ^{2}}{\Delta \varphi
-\left( 1-\frac{2}{n}\right) \left( 1+\frac{1}{\varepsilon }\right)
\left\vert \nabla \varphi \right\vert ^{2}+2\lambda _{\mathbb{S}}}%
e^{-\varphi }\cdot \Vert T^{\ast }s\Vert _{H_{1}}^{2}
\end{equation*}

Lemma \ref{Lemma-Hormander} implies that there exists some $u_{\Omega }\in 
\mathrm{Dom}(T)\subseteq L_{\varphi }^{2}(\Omega ,\mathbb{S})$ such that%
\begin{equation}
Du_{\Omega }=f_{\Omega }\ \mathrm{and}\ \int_{\Omega }\left\vert u_{\Omega
}\right\vert ^{2}e^{-\varphi }\leq \frac{1}{C}\int_{\Omega }\frac{\left\vert
f_{\Omega }\right\vert ^{2}}{\Delta \varphi -\left( 1-\frac{2}{n}\right)
\left( 1+\frac{1}{\varepsilon }\right) \left\vert \nabla \varphi \right\vert
^{2}+2\lambda _{\mathbb{S}}}e^{-\varphi }.  \label{Du-omega}
\end{equation}

Given $f\in L_{\varphi }^{2}(M,\mathbb{S})$ with $\int_{M}\frac{\left\vert
f\right\vert ^{2}}{\Delta \varphi -\left( 1-\frac{2}{n}\right) \left( 1+%
\frac{1}{\varepsilon }\right) \left\vert \nabla \varphi \right\vert
^{2}+2\lambda _{\mathbb{S}}}e^{-\varphi }<\infty $, we can finish the proof
by applying $\left( \ref{Du-omega}\right) $ to an increasing sequence of
open subsets $\Omega _{0}\subset \subset \Omega _{1}\subset \subset \cdots
\nearrow M$ to get a sequence of $u_{\nu }\in L_{\varphi }^{2}(\Omega _{\nu
},\mathbb{S})$ such that $Du_{\nu }=f|_{\Omega _{\nu }}$ in the sense of
distribution and 
\begin{equation*}
\int_{\Omega _{\nu }}\left\vert u_{\nu }\right\vert ^{2}e^{-\varphi }\leq 
\frac{1}{C}\int_{\Omega _{\nu }}\frac{\left\vert f\right\vert ^{2}}{\Delta
\varphi -\left( 1-\frac{2}{n}\right) \left( 1+\frac{1}{\varepsilon }\right)
\left\vert \nabla \varphi \right\vert ^{2}+2\lambda _{\mathbb{S}}}%
e^{-\varphi },\nu =0,1,\cdots .
\end{equation*}%
The above uniform $L^{2}$-estimate allows us obtain a desired solution of $%
Du=f$ by taking a weak limit of $\{u_{\nu }\}$. The proof is
complete.\bigskip
\end{proof}

\begin{remark}
\bigskip Different from the Bochner formula of the $\overline{\partial }$%
-operator, for the Dirac operator $D$, only $\left\vert Ds\right\vert ^{2}$
is involved on the left side of $\left( \ref{L2-identity}\right) $. This
simplifies the $L^{2}$-estimates in our setting compared to \cite{H1}, as we
only need the minimal extension of the Dirac operator $D$, for which the set
of compactly supported smooth sections is dense in its domain with respect
to the graph norm.
\end{remark}

\section{The first eigenvalue of Dirac operators\label{2d}}

\bigskip By constructing the weight function $\varphi $ in Theorem \ref%
{RieSurfaceCase} from certain Poisson equation, we obtain the following

\begin{corollary}
\label{eigen-by-curvature-integral}If $M$ is a closed $2$-dimensional
Riemannian manifold, then 
\begin{equation}
\lambda _{\mathrm{min}}(D^{2})\geq \frac{2}{\mathrm{Vol}(M)}\int_{M}\lambda
_{\mathbb{S}},  \label{eigen-RieSfc}
\end{equation}%
where $\lambda _{\mathrm{min}}(D^{2})$ is the first eigenvalue of $D^{2}$.
Consequently, if 
\begin{equation}
\int_{M}\lambda _{\mathbb{S}}>0,  \label{positive-integral-condition}
\end{equation}%
then $D:L^{2}(M,\mathbb{S})\rightarrow L^{2}(M,\mathbb{S})$ defines an
isomorphism.
\end{corollary}

\begin{proof}
Since $\lambda _{\min }\left( D^{2}\right) \geq 0$, the inequality $\left( %
\ref{eigen-RieSfc}\right) $ is trivial if $\int_{M}\lambda _{\mathbb{S}}\leq
0$. Now let us assume $\int_{M}\lambda _{\mathbb{S}}>0$. The equation $%
\Delta \varphi +2\lambda _{\mathbb{S}}=\frac{2}{\mathrm{Vol}\left( M\right) }%
\int_{M}\lambda _{\mathbb{S}}$ always has a solution $\varphi $, as the
integral of $\lambda _{\mathbb{S}}-\frac{1}{\text{Vol}\left( M\right) }%
\int_{M}\lambda _{\mathbb{S}}$ over $M$ is zero. (In fact, $\Delta \varphi
+2\lambda _{\mathbb{S}}>0$ has a solution $\varphi $ if and only if $%
\int_{M}\lambda _{\mathbb{S}}>0$). Take such $\varphi $ as the weight
function in $\left( \ref{weight-L2-solvability}\right) $. For any $f\in
L^{2}\left( M,\mathbb{S}\right) $, by our condition 
\begin{equation*}
\int_{M}\frac{\left\vert f\right\vert ^{2}}{\Delta \varphi +2\lambda _{%
\mathbb{S}}}e^{-\varphi }=\frac{\mathrm{Vol}\left( M\right) }{%
2\int_{M}\lambda _{\mathbb{S}}}\int_{M}\left\vert f\right\vert
^{2}e^{-\varphi }<\infty .
\end{equation*}%
By Theorem \ref{RieSurfaceCase}, there is a $L^{2}$ section $u$ such that $%
Du=f$. From $\left( \ref{weight-L2-solvability}\right) $ we have%
\begin{equation}
\frac{\int_{M}\left\vert Du\right\vert ^{2}e^{-\varphi }}{\int_{M}\left\vert
u\right\vert ^{2}e^{-\varphi }}\geq \frac{2}{\text{$\mathrm{Vol}$}\left(
M\right) }\int_{M}\lambda _{\mathbb{S}}\text{.}  \label{Reyleigh-weighted}
\end{equation}%
By the Rayleigh quotient, the desired estimate $\lambda ^{2}\geq \frac{2}{%
\text{\textrm{Vol}}\left( M\right) }\int_{M}\lambda _{\mathbb{S}}$ follows,
where $\lambda $ is a square root of $\lambda _{\min }\left( D^{2}\right) $.
(To get rid of the weight $e^{-\varphi }$ in $\left( \ref{Reyleigh-weighted}%
\right) $, we consider the operator $\widetilde{D}=e^{-\frac{\varphi }{2}%
}\circ D\circ e^{\frac{\varphi }{2}}$ and section $\widetilde{u}=e^{-\frac{%
\varphi }{2}}u$, and notice $D^{2}$ and $\widetilde{D}^{2}$ have the same
eigenvalues).
\end{proof}

\bigskip

\begin{remark}
Corollary \ref{eigen-by-curvature-integral} also follows from \bigskip $%
\left( \ref{weighted-L2-est2}\right) $, by taking the same $\varphi $ as
above and the nontrivial section $s$ such that $D\left( e^{-\varphi
}s\right) =\lambda e^{-\varphi }s$.\bigskip\ A different proof was given in
Theorem 2 of \cite{B}.
\end{remark}

\begin{proposition}
\label{Noncpt-2d-mfds}If $M$ is a noncompact $2$-dimensional Riemannian
manifold, then for any $f\in L_{loc}^{2}(M,\mathbb{S})$, there exists a
section $u\in L_{loc}^{2}(M,\mathbb{S})$ such that $Du=f$, where $%
L_{loc}^{2}(M,\mathbb{S})$ is the space of locally square integrable
sections of $\mathbb{S}$.
\end{proposition}

\begin{proof}
By Theorem \ref{RieSurfaceCase}, it suffices to prove that there is a $%
\varphi \in C^{2}(M)$ such that 
\begin{equation}
\Delta \varphi +2\lambda _{\mathbb{S}}\geq 0\text{ on }M\text{ and}\int_{M}%
\frac{\left\vert f\right\vert ^{2}}{\Delta \varphi +2\lambda _{\mathbb{S}}}%
e^{-\varphi }<\infty .  \label{phi-R-condition}
\end{equation}

We first construct some nonnegative proper exhaustion function $\psi \in
C^{2}(M)$ such that $\Delta \psi +2\lambda _{\mathbb{S}}\geq 1$ on $M$.
Since $M$ is a noncompact $2$-dimensional Riemannian manifold, there always
exists a nonnegative exhaustion function $\phi \in C^{\infty }(M)$ which is
strictly subharmonic. Then, we choose a nonnegative function $\kappa \in
C^{\infty }[0,+\infty )$ such that 
\begin{equation}
\kappa ^{\prime }(t)>0,\kappa ^{\prime \prime }(t)\geq 0\ \mathrm{for}\
t\geq 0,\ \kappa ^{\prime }(\nu )\geq \sup_{\Omega _{\nu +1}\setminus \Omega
_{\nu }}\frac{1-2\lambda _{\mathbb{S}}}{\Delta \phi }\ \mathrm{for}\ \nu
=0,1,2\cdots ,  \label{k-p}
\end{equation}%
where $\Omega _{\nu }:=\{x\in M\ \mid \ \phi (x)<\nu \}(\nu =0,1,2,\cdots )$%
. Set $\psi =\kappa \circ \phi $, then 
\begin{equation*}
\Delta \psi =\kappa ^{\prime }\circ \phi \cdot \Delta \phi +\kappa ^{\prime
\prime }\circ \phi \cdot |\nabla \phi |^{2}\geq \kappa ^{\prime }\circ \phi
\cdot \Delta \phi .
\end{equation*}%
Consequently, by the monotonicity of $\kappa ^{\prime }$ and $\left( \ref%
{k-p}\right) $, we obtain $\Delta \psi +2\lambda _{\mathbb{S}}\geq 1$ on $M$%
. Since $\kappa (t)\rightarrow +\infty $ as $t\rightarrow +\infty $, $\psi
=\kappa \circ \phi $ is also an exhaustion function.

Now we construct the desired function $\varphi $ satisfying $\left( \ref%
{phi-R-condition}\right) $. If we set $\Omega _{\nu }=\psi ^{-1}(-\infty
,\nu ),\nu =0,1,2,\cdots ,$ then $\varnothing =\Omega _{0}\subset \subset
\Omega _{1}\subset \subset \Omega _{2}\subset \subset \cdots \nearrow M$.
Let $\chi \in C^{\infty }[0,+\infty )$ such that 
\begin{equation*}
\chi \left( \nu \right) \geq \nu +\log \int_{\Omega _{\nu +1}}\left\vert
f\right\vert ^{2}\text{ }\left( \nu =0,1,2,\cdots \right) ,\text{ }\chi
^{\prime }\geq 1,\text{ }\chi ^{\prime \prime }\geq 0.
\end{equation*}%
Define $\varphi =\chi \circ \psi $, then we have

\begin{equation*}
\Delta \varphi =\chi ^{\prime }\circ \psi \cdot \Delta \psi +\chi ^{\prime
\prime }\circ \psi \cdot \left\vert \nabla \psi \right\vert ^{2}\geq \Delta
\psi \geq 1-2\lambda _{\mathbb{S}},
\end{equation*}%
and%
\begin{eqnarray}
\int_{M}\frac{\left\vert f\right\vert ^{2}}{\Delta \varphi +2\lambda _{%
\mathbb{S}}}e^{-\varphi } &=&\sum_{\nu \geq 0}\int_{\Omega _{\nu
+1}\setminus \Omega _{\nu }}\frac{\left\vert f\right\vert ^{2}}{\Delta
\varphi +2\lambda _{\mathbb{S}}}e^{-\varphi }  \notag \\
&\leq &\sum_{\nu \geq 0}e^{-\chi (\nu )}\int_{\Omega _{\nu +1}\setminus
\Omega _{\nu }}\left\vert f\right\vert ^{2}  \notag \\
&\leq &\sum_{\nu \geq 0}e^{-\nu }<\infty .\   \notag
\end{eqnarray}%
Hence, we have constructed the desired the weight function $\varphi $
satisfying $\left( \ref{phi-R-condition}\right) $ and the proof is therefore
complete.
\end{proof}

\begin{corollary}
The Poisson equation $\Delta _{d}u=f$ is always solvable on noncompact $2$%
-dimensional Riemannian manifolds (we do not require $M$ is orientable).
\end{corollary}

\begin{proof}
For any $2$-dimensional Riemannian manifold $(M,g)$, the vector bundle $%
\Lambda ^{\ast }M:=\displaystyle{\bigoplus_{\ell =0}^{2}}\Lambda ^{\ell
}(T^{\ast }M)$ has a natural structure of a Dirac bundle over $(M,g)$, and
the associated Dirac operator is given by $D=d+\delta $, where $\delta $ is
the codifferential operator. Since $D^{2}=(d+\delta )^{2}=\Delta _{d}$ (the
Hodge Laplacian operator), by the above proposition, the corollary follows.
\end{proof}

\bigskip \bigskip By a suitable choice of the weight function in our
estimate (\ref{weighted-L2-est1}) , we can give a simple proof of B\"{a}r's

\begin{theorem}
(Theorem 3 \cite{B})\label{first-eigen-est} Let $\mathbb{S}$ be a Dirac
bundle over a compact $n$-dimensional Riemannian manifold $\left( M,g\right) 
$ without boundary, and $D$ be the Dirac operator, $n\geq 2$. Then 
\begin{equation*}
\lambda _{\mathrm{min}}(D^{2})\geq \frac{n}{n-1}\lambda _{\mathrm{min}}(%
\mathbb{L})
\end{equation*}%
where $\lambda _{\mathrm{min}}(\cdot )$ means the first eigenvalue, $\mathbb{%
L}=-\frac{n-1}{n-2}\Delta +\lambda _{\mathbb{S}}$ if $n\geq 3$, and $\mathbb{%
L}=-\frac{1}{2}\Delta +\lambda _{\mathbb{S}}$ if $n=2$.
\end{theorem}

\begin{proof}
For any $\varphi \in C^{\infty }(M)$, we can choose a non-trivial section $%
s\in \Gamma (M,\mathbb{S})$ such that 
\begin{equation*}
D(e^{-a\varphi }s)=\lambda e^{-a\varphi }s
\end{equation*}%
where $\lambda $ is a square root of $\lambda _{\mathrm{min}}(D^{2})$ and $%
a=0$ if $n=2$, $a=\frac{n}{2(n-1)}$ if $n\geq 3.$ By Lemma \ref%
{weight-relation}, we know 
\begin{equation*}
D_{\varphi }^{\ast }s=D_{a\varphi }^{\ast }s+(a-1)\nabla \varphi \cdot
s=\lambda s+(a-1)\nabla \varphi \cdot s.
\end{equation*}%
Substituting the above identity into (\ref{weighted-L2-est1}) and noticing $%
\mathrm{Re}\left\langle s,\nabla \varphi \cdot s\right\rangle =0$, we have 
\begin{equation}
\lambda ^{2}\int_{M}|s|^{2}e^{-\varphi }\geq \int_{M}\left( \Delta \varphi
-|\nabla \varphi |^{2}+2\lambda _{\mathbb{S}}\right) |s|^{2}e^{-\varphi },\ 
\mathrm{if}\ n=2,  \label{n = 2}
\end{equation}%
and 
\begin{equation}
\lambda ^{2}\int_{M}|s|^{2}e^{-\varphi }\geq \frac{n}{n-1}\int_{M}\left( 
\frac{1}{2}\Delta \varphi -\frac{n-2}{4(n-1)}|\nabla \varphi |^{2}+\lambda _{%
\mathbb{S}}\right) |s|^{2}e^{-\varphi },\ \mathrm{if}\ n\geq 3.
\label{n geq 3}
\end{equation}

Let $\psi \in C^{\infty }(M)$ be an eigenfunction 
\begin{equation*}
\mathbb{L}\psi =\lambda _{\mathrm{min}}(\mathbb{L})\psi .
\end{equation*}%
Without loss of generality, we may assume $\psi >0$ on $M$. Set 
\begin{equation*}
\varphi =-\log \psi \ \mathrm{if}\ n=2\text{, and }\varphi =-\frac{2(n-1)}{%
n-2}\log \psi \ \mathrm{if}\ n\geq 3,
\end{equation*}%
then Theorem \ref{first-eigen-est} follows form (\ref{n = 2}) and (\ref{n
geq 3}).
\end{proof}

\bigskip

An example of a Dirac bundle over a $3$-manifold is the normal bundle of an
instanton in a $G_{2}$ manifold. In Physics, $G_{2}$-manifolds are internal
spaces for compactification in M-theory in eleven dimensional spacetimes,
similar to the role of Calabi-Yau threefolds in string theory. Counting
instantons in $G_{2}$ manifolds is similar to counting holomorphic curves in
Calabi-Yau threefolds.

\begin{definition}
\bigskip A $G_{2}$ manifold is a $7$-dimensional Riemannian manifolds $%
\left( M,g\right) $ equipped with a parallel cross product $\times $. An
instanton (or associative submanifold) $A$ is a $3$-dimensional submanifold
whose tangent spaces are closed under the cross product.
\end{definition}

Let $N_{A/M}$ be the normal bundle of $A$ in $M$. Regarding $N_{A/M}$ as a
left Clifford module over $A$ with the $G_{2}$ cross product $\times $ as
the Clifford multiplication, it is a\emph{\ twisted spinor bundle }over $A$,
with the normal connection $\nabla ^{\perp }$ inherited from the Levi-Civita
connection $\nabla $ on $M$ (Section 5 \cite{M}). All $3$-manifolds nearby $%
A $ can be parameterized by sections $V$ of $N_{A/M}$.

Given the cross product $\times $, one can defined a $TM$-valued $3$-form $%
\tau $ on $M$ as 
\begin{equation*}
\tau (u,v,w)=-u\times \left( v\times w\right) -g(u,v)w+g(u,w)v,
\end{equation*}
for $u,v$ and $w\in TM$. Then $A$ is an instanton if and only if $\tau
|_{A}=0$ (c.f. \cite{HL}). Using $\tau $ McLean (Section 5 \cite{M}) defined
a nonlinear function 
\begin{equation*}
F:C^{1,\alpha }\left( A,N_{A/M}\right) \rightarrow C^{\alpha }\left(
A,N_{A/M}\right) \text{ }\left( 0<\alpha <1\right)
\end{equation*}%
such that instantons nearby $A$ correspond to the zeros of $F$ (the choice
of the \emph{Schauder} functional analysis setting over the $W^{1,p}$
setting is necessary, due to the \emph{cubic nonlinearity} of $\tau $ and $F$%
). He computed 
\begin{equation}
\left. \frac{d}{dt}\right\vert _{t=0}F\left( tV\right) =DV\text{,}
\label{linearization}
\end{equation}%
where $V\in $ $A,N_{A/M}$, and $D$ is the \emph{twisted Dirac operator} on $%
N_{A/M}$ over $A$. Then he proved

\begin{theorem}
\label{deform-instanton}(Theorem 5-2 \cite{M}) Infinitesimal deformations of
instantons at $A$ are parametrized by the space of harmonic twisted spinors
on $A$, i.e. the kernel of $D$.
\end{theorem}

\bigskip More detailed exposition of McLean's proof can be found in Theorem
9 and Section 4.2 of \cite{LWZ}.

We relate Theorem \ref{first-eigen-est} to \emph{rigidity }of instantons in $%
G_{2}$ manifolds, i.e. situation that the moduli space of instantons near $A$
is a zero dimensional smooth manifold.

\begin{corollary}
\label{vanishing-rigid}If an instanton $A$ is compact and $\lambda _{\min
}\left( \mathbb{L}\right) >0$, then $A$ is rigid. Here $\mathbb{L=-}2\Delta
_{\left( A,g\right) }+\lambda _{N_{A/M}}$, and $\Delta _{\left( A,g\right) }$
is the Laplace-Beltrami operator on $A$ with the induced metric from $\left(
M,g\right) $.
\end{corollary}

\begin{proof}
Let $\mathbb{S}$ be $N_{A/M}$ and $D$ be the twisted Dirac operator in
Theorem \ref{first-eigen-est}. By our condition the kernel and cokernel of $%
D:W^{1,2}\left( A,N_{A/M}\right) \rightarrow L^{2}\left( A,N_{A/M}\right) $
vanish (note $D$ is self-adjoint). A standard interpolation argument implies
that $D:W^{1,p}\left( A,N_{A/M}\right) \rightarrow L^{p}\left(
A,N_{A/M}\right) $ ($p>3$) is surjective (e.g. estimates between (36)$\sim $%
(37) in \cite{LWZ}), which in turn implies that $D:$ $C^{1,\alpha }\left(
A,N_{A/M}\right) \rightarrow C^{\alpha }\left( A,N_{A/M}\right) $ is
surjective by the Schauder estimate of $D$ and Sobolev embedding $%
W^{1,p}\hookrightarrow C^{0}$. By the implicit function theorem for the
nonlinear function $F$, the moduli space $F^{-1}\left( 0\right) $ near $A$
is a zero dimensional smooth manifold, and $A$ is rigid.
\end{proof}

Clearly if $\lambda _{N_{A/M}}>0$ everywhere on $M$, then $\lambda _{\min
}\left( \mathbb{L}\right) >0$, but not vice versa. So Corollary \ref%
{vanishing-rigid} provides a potentially weaker condition to guarantee the
rigidity of $A$.\bigskip\ The computation of $\mathfrak{R}_{N_{A/M}}$ in
terms of the curvature of $M$ and the second fundamental form of $A$ can be
found in Section 5.3 of \cite{G}.

\section{ $\mathbb{Z}_{2}$-graded Dirac operators\label{Z2-Dirac}}

\bigskip The solvability of the \textquotedblleft half\textquotedblright\
Dirac equation $D^{\pm }u=f$ \ is of interest for several reasons: e.g. in
quaternionic analysis, the correct generalization of analytical functions on 
$\mathbb{C}$ are solutions of $D^{+}u=0$ on $\mathbb{H}$, as requiring $u$
to be quaternion differentiable only yields linear functions; $D^{\pm }$ may
have nonzero Fredholm index to produce nontrivial invariants; $D^{\pm }$
arises as the linearized operator (modulo zeroth order terms) in many moduli
problems, like those of $J$-holomorphic curves and solutions of the
Seiberg-Witten equation. As we will see, consideration of the $\mathbb{Z}%
_{2} $-grading also improves eigenvalue estimates of the Dirac operator $D$
on even dimensional manifolds.

Our main technical tool, Proposition \ref{weighted-L2-prop}, extends
immediately to $D^{\pm }$:

\begin{proposition}
\label{weighted-egen-pm}For any smooth section $s$ of $\mathbb{S}^{\mp }$
with compact support and any $C^{2}$ function $\varphi :M\rightarrow \mathbb{%
R}$, we have 
\begin{eqnarray*}
&&\frac{n-1}{n}\int_{M}\left\vert (D^{\pm })_{\varphi }^{\ast }s\right\vert
^{2}e^{-\varphi }+\frac{n-2}{n}\mathrm{Re}\int_{M}\left\langle \nabla
\varphi \cdot s,(D^{\pm })_{\varphi }^{\ast }s\right\rangle e^{-\varphi } \\
&\geq &\int_{M}\left[ \frac{1}{2}\Delta \varphi -\left( \frac{1}{2}-\frac{1}{%
n}\right) \left\vert \nabla \varphi \right\vert ^{2}+\lambda _{\mathbb{%
S^{\mp }}}\right] \left\vert s\right\vert ^{2}e^{-\varphi }.
\end{eqnarray*}%
and 
\begin{equation*}
\int_{M}\left\vert (D^{\pm })_{\varphi }^{\ast }s\right\vert ^{2}e^{-\varphi
}\geq C\int_{M}\left[ \Delta \varphi -\left( 1-\frac{2}{n}\right) \left( 1+%
\frac{1}{\varepsilon }\right) \left\vert \nabla \varphi \right\vert
^{2}+2\lambda _{\mathbb{S^{\mp }}}\right] \left\vert s\right\vert
^{2}e^{-\varphi }.
\end{equation*}%
where $C$ is the constant in Proposition \ref{weighted-L2-prop}.
\end{proposition}

\begin{proof}
Consider the Dirac operator $D^{+}:L_{\varphi }^{2}\left( M,\mathbb{S}%
^{+}\right) \rightarrow L_{\varphi }^{2}\left( M,\mathbb{S}^{-}\right) $
(the $D^{-}$ case is similar). Let $\left( D^{+}\right) _{\varphi }^{\ast }$
be its formal adjoint with respect to the measure $e^{-\varphi }dvol_{g}$.
Then it is easy to observe%
\begin{equation}
\left( D^{+}\right) _{\varphi }^{\ast }s=e^{\varphi }D^{-}\left( e^{-\varphi
}s\right) =e^{\varphi }D\left( e^{-\varphi }s\right) =-\nabla \varphi \cdot
s+D^{-}s  \label{weighted-conjugate}
\end{equation}%
for smooth sections $s$ of $\mathbb{S}^{-}$, where we have used that $%
D^{-}s=D|_{\mathbb{S}^{-}}\left( s\right) $.

When we restrict the sections from $\Gamma \left( \mathbb{S}\right) $ to $%
\Gamma \left( \mathbb{S}^{-}\right) $, by $\left( \ref{pos-neg-Dirac}\right) 
$ and $\left( \ref{curvature-pm}\right) $, $D^{2}$ becomes $D^{+}D^{-}$, and 
$\mathfrak{R}$ becomes $\mathfrak{R}^{-}$ in the Bochner formula $\left( \ref%
{Gen-Bochner}\right) $, i.e. 
\begin{equation*}
D^{+}D^{-}=\nabla ^{\ast }\nabla +\mathfrak{R}^{-}.
\end{equation*}%
Similarly%
\begin{equation*}
D^{-}D^{+}=\nabla ^{\ast }\nabla +\mathfrak{R}^{+}.
\end{equation*}%
Integrating on $M$ we obtain%
\begin{equation}
\int_{M}\left\vert D^{\pm }s\right\vert ^{2}=\int_{M}\left\vert \nabla
s\right\vert ^{2}+\int_{M}\left\langle \mathfrak{R}^{\pm }s,s\right\rangle 
\text{.}  \label{Integral-Bochner}
\end{equation}

The remaining part of the proof is the same as Proposition \ref%
{weighted-L2-prop}, except that $\lambda _{\mathbb{S}}$ is replaced by $%
\lambda _{\mathbb{S}^{-}}$, where \ 
\begin{equation}
\lambda _{\mathbb{S}^{\pm }}\left( x\right) :=\text{the smallest eigenvalue
of }\mathfrak{R}^{\pm }\left( x\right)  \label{lambda-s-pm}
\end{equation}%
at any $x\in M$. The proposition follows.\ \ \ \ \ 
\end{proof}

\begin{remark}
\label{D^pm} From Proposition \ref{weighted-egen-pm}, we know that
Proposition \ref{Dirac-solve} still holds if we replace simultaneously $D$
by $D^{\pm }$ and $\lambda _{\mathbb{S}}$ by $\lambda _{\mathbb{S^{\mp }}}.$
\end{remark}

\bigskip

From Proposition \ref{weighted-egen-pm} and Remark \ref{D^pm}, similarly we
obtain the results of $D^{\pm }$ parallel to Theorem \ref{RieSurfaceCase},
Corollary \ref{eigen-by-curvature-integral} and Theorem \ref%
{cylindrical-solvability}, by replacing $D$ by $D^{\pm }$, and $\lambda _{%
\mathbb{S}}$ by $\lambda _{\mathbb{S}^{\mp }}$ in the corresponding
statements. For example, we can refine Corollary \ref%
{eigen-by-curvature-integral} to Corollary \ref{curvature-integral-D+}. From
this we obtain Corollary \ref{curvature-integral-dbar} as follows.

\medskip

\begin{proof}
(of Corollary \ref{curvature-integral-dbar}). Let $\mathbb{S^{+}=}E$ and $%
\mathbb{S^{-}=\wedge }^{0,1}\left( E\right) $ over the Riemann surface $M$, $%
\overline{\partial }:E\rightarrow $ $\mathbb{\wedge }^{0,1}\left( E\right) $
be the Cauchy-Riemann operator, and $\overline{\partial }^{\ast }$ be its
adjoint. Let $D^{+}=\sqrt{2}\overline{\partial },$ $D^{-}=\sqrt{2}\overline{%
\partial }^{\ast }$, then $D=\left( D^{+},D^{-}\right) $ is the
Dolbeault-Dirac operator on $\mathbb{S=S^{+}\oplus S^{-}}$. From Lemma \ref%
{curve. term Spinc}, we have the curvature operators $\mathfrak{R}^{\pm }$
of $\mathbb{S^{\pm }}$ as%
\begin{equation*}
\mathfrak{R}^{-}=\sqrt{-1}\Lambda R^{E}+K,\ \ \ \ \mathfrak{R}^{+}=-\sqrt{-1}%
\Lambda R^{E}.
\end{equation*}%
By definitions $\left( \ref{thetaE}\right) $ and $\left( \ref{ThetaE}\right) 
$, we have%
\begin{eqnarray*}
\lambda _{\mathbb{S^{+}}} &=&\text{the smallest eigenvalue of }-\sqrt{-1}%
\Lambda R^{E}=-\Theta _{E}, \\
\lambda _{\mathbb{S^{-}}} &=&\text{the smallest eigenvalue of }\sqrt{-1}%
\Lambda R^{E}+K=\theta _{E}+K.
\end{eqnarray*}%
Applying Corollary \ref{curvature-integral-D+} to $D$, we obtain Corollary %
\ref{curvature-integral-dbar}.
\end{proof}

By Corollary \ref{curvature-integral-D+}, we improve the estimate for $%
\lambda _{\mathrm{min}}(D^{2})$ in Corollary \ref%
{eigen-by-curvature-integral} as follows.

\begin{corollary}
\label{eigen-est-surface-Z2}Under the hypothesis in Corollary \ref%
{curvature-integral-D+}, it holds that 
\begin{equation*}
\lambda _{\mathrm{min}}(D^{2})\geq \frac{2}{\mathrm{Vol}(M)}\min \left\{
\int_{M}\lambda _{\mathbb{S^{+}}},\int_{M}\lambda _{\mathbb{S^{-}}}\right\} .
\end{equation*}
\end{corollary}

\bigskip From Proposition \ref{weighted-egen-pm}, we can also improve
Theorem \ref{first-eigen-est} of B\"{a}r. As in the proof of Theorem \ref%
{first-eigen-est}, let $s\in \Gamma (M,\mathbb{S})$ be a non-trivial section
such that $D(e^{-a\varphi }s)=\lambda e^{-a\varphi }s$, which implies 
\begin{equation*}
(D^{\pm })_{a\varphi }^{\ast }s^{\mp }=\lambda s^{\pm }.
\end{equation*}%
By applying Proposition \ref{weighted-egen-pm} to $s^{\mp }$ in the same way
as in the proof of Theorem \ref{first-eigen-est}, we have the following

\begin{corollary}
\label{Eigen-improve-Z2_Dirac}Let $\mathbb{S}$ be a $\mathbb{Z}_{2}$-graded
Dirac bundle over a compact $n$-dimensional Riemannian manifold $\left(
M,g\right) $ without boundary, and $D$ be the Dirac operator, $n\geq 2$.
Then 
\begin{equation*}
\lambda _{\mathrm{min}}(D^{2})\geq \frac{n}{n-1}\min \left\{ \lambda _{%
\mathrm{min}}(\mathbb{L^{+}}),\lambda _{\mathrm{min}}(\mathbb{L^{-}})\right\}
\end{equation*}%
where $\mathbb{L^{\pm }}=-\frac{n-1}{n-2}\Delta +\lambda _{\mathbb{S^{\pm }}%
} $ if $n\geq 3$, and $\mathbb{L^{\pm }}=-\frac{1}{2}\Delta +\lambda _{%
\mathbb{S^{\pm }}}$ if $n=2$.
\end{corollary}

\section{Manifolds with cylindrical ends\label{cyl}}

\bigskip Between the compact and noncompact cases, there is an important
case of manifolds with cylindrical ends. There are many works of
differential operators on such manifolds, going back to the work of Lockhart
and McOwen (\cite{LoMa}), and occurring often in gauge theory and low
dimensional topology (e.g. \cite{Don}, \cite{T}).

\begin{definition}
\label{cylindrical}Let $\left( M,g\right) $ be a $n$-dimensional Riemannian
manifold. We say it is a \emph{manifold with cylindrical ends} if outside a
compact subset, $M$ consists of product Riemannian manifolds $E_{\nu }$ $%
\simeq \lbrack 0,+\infty )\times B_{\nu }$ ($\nu =1,2,\cdots m$), where each 
$B_{\nu }$ is a $\left( n-1\right) $-dimensional compact Riemannian
manifold. Each $E_{\nu }$ is called a \emph{cylindrical end}, and the $%
[0,+\infty )$ direction is called the \emph{cylindrical direction}.
\end{definition}

\bigskip For analysis on cylindrical manifolds, one often needs the \emph{%
Sobolev space with} \emph{exponential weights}. In other words, one chooses
a smooth weight function $\varphi $ on $M$ such that 
\begin{equation}
\varphi |_{E_{\nu }}=-\delta _{\nu }\tau _{\nu }\text{ for some constant }%
\delta _{\nu }\geq 0\text{, and large }\tau _{\nu }\text{ }
\label{exp-weight}
\end{equation}%
$\left( \nu =1,2,\cdots m\right) $, and then defines the \emph{weighted
Sobolev norm} 
\begin{equation}
\left\Vert f\right\Vert _{W_{\delta }^{k,p}}:=\left\Vert e^{-\frac{\varphi }{%
p}}f\right\Vert _{W^{k,p}},  \label{weighted-Sobolev}
\end{equation}%
where $\delta =\left( \delta _{1},\cdots ,\delta _{m}\right) $. Different
choices of $\varphi $ satisfying $\left( \ref{exp-weight}\right) $ only
result in equivalent Banach spaces.

\bigskip

\begin{proof}
(of Theorem \ref{cylindrical-solvability}). Without loss of generality we
assume $M\backslash K=\cup _{\nu =1}^{m}E_{\nu }$, where $E_{\nu }$ are the
cylindrical ends of $M$, (otherwise we can always enlarge $K)$. We choose a
smooth cut-off function $\rho :M\rightarrow \left[ 0,1\right] $ such that $%
\rho \equiv 0$ on $K$, and on $E_{\nu }$ $\rho $ is a function of\ the
variable $\tau _{\nu \text{ }}$in the cylindrical direction satisfying 
\begin{equation}
\rho \left( \tau _{\nu}\right) =\left\{ 
\begin{array}{c}
0\text{, if }0\leq \tau _{\nu }\leq 1 \\ 
1\text{, if }\tau _{\nu }\geq 2%
\end{array}%
\right. \text{, and }0\leq \rho ^{\prime }\left( \tau _{\nu }\right) \text{, 
}\left\vert \rho ^{\prime \prime }\left( \tau _{\nu }\right) \right\vert
\leq 2\text{.}  \notag
\end{equation}%
Let $\mu >0$ be the first eigenvalue of the Dirichlet eigenvalue problem 
\begin{equation}
-\Delta \eta =\mu \eta \text{ on }M\backslash \cup _{\nu =1}^{m}\left(
3,\infty \right) \times B_{\nu }.  \label{eigen-equ}
\end{equation}%
By the nodal domain theorem, we can find an eigenfunction function $\eta >0$
on $M\backslash \cup _{\nu =1}^{m}\left( 3,\infty \right) \times B_{\nu }$.
Setting $A=\frac{1}{\left( 1-\frac{2}{n}\right) \left( 1+\frac{1}{%
\varepsilon }\right) }>0$, from $\left( \ref{eigen-equ}\right) $ it follows
that for sufficiently large $\varepsilon >0$%
\begin{eqnarray}
\left( -\frac{1}{\left( 1-\frac{2}{n}\right) \left( 1+\frac{1}{\varepsilon }%
\right) }\Delta +2\lambda _{\mathbb{S}}\right) \eta ^{\gamma } &=&\eta
^{\gamma -2}\left[ -A\gamma \eta \Delta \eta -A\gamma \left( \gamma
-1\right) \left\vert \nabla \eta \right\vert ^{2}+2\lambda _{\mathbb{S}}\eta
^{2}\right]  \notag \\
&\geq &\eta ^{\gamma -1}\left[ -A\gamma \Delta \eta +2\lambda _{\mathbb{S}%
}\eta \right]  \notag \\
&\geq &\eta ^{\gamma }\left[ A\gamma \mu -2\beta \right] >0
\label{eigen-pos}
\end{eqnarray}
on $M\backslash \cup _{\nu =1}^{m}[2,\infty )\times B_{\nu }$, provided that 
$\beta $ satisfies the condition%
\begin{equation*}
0<\beta <\frac{\gamma \mu }{2-\frac{4}{n}},
\end{equation*}%
where $\gamma \in (0,1)$ is a constant to be determined. Let 
\begin{equation*}
\phi :=-\frac{1}{\left( 1-\frac{2}{n}\right) \left( 1+\frac{1}{\varepsilon }%
\right) }\log \eta ,
\end{equation*}%
then on $M\backslash \cup _{\nu =1}^{m}[2,\infty )\times B_{\nu }$, by $%
\left( \ref{eigen-pos}\right) $ we have 
\begin{eqnarray}
&&\Delta \left( \gamma \phi \right) -\left( 1-\frac{2}{n}\right) \left( 1+%
\frac{1}{\varepsilon }\right) \left\vert \nabla \left( \gamma \phi \right)
\right\vert ^{2}+2\lambda _{\mathbb{S}}  \notag \\
&=&\eta ^{-\gamma }\left[ -\frac{1}{\left( 1-\frac{2}{n}\right) \left( 1+%
\frac{1}{\varepsilon }\right) }\Delta \eta ^{\gamma }+2\lambda _{\mathbb{S}%
}\eta ^{\gamma }\right] >0.  \label{gama-fi}
\end{eqnarray}

We define a function $h:\cup _{\nu =1}^{m}E_{\nu }\rightarrow \mathbb{R}$,
such that 
\begin{equation}
h\left( \tau _{\nu },b_{\nu }\right) =-\delta _{\nu }\tau _{\nu },  \label{h}
\end{equation}%
where the constants $\delta _{\nu }\geq 0$ are to be determined. Then we
define the weight function $\varphi :M\rightarrow \mathbb{R}$ as%
\begin{equation}
\varphi =\gamma \left( 1-\rho \right) \phi +\rho h.  \label{b}
\end{equation}%
It is easy to check that $\varphi $ is smooth and is globally defined on $M$.

By $\left( \ref{gama-fi}\right) \sim \left( \ref{b}\right) $, for
sufficiently small $\delta _{1},\cdots ,\delta _{m}\geq 0$ and sufficiently
large $\varepsilon >0$ , we have 
\begin{eqnarray*}
&&\Delta \varphi -\left( 1-\frac{2}{n}\right) \left( 1+\frac{1}{\varepsilon }%
\right) \left\vert \nabla \varphi \right\vert ^{2}+2\lambda _{\mathbb{S}} \\
&&\left\{ 
\begin{array}{c}
\text{ }>0\text{ \ \ \ \ \ \ \ \ \ \ \ \ \ \ \ \ \ \ \ \ \ \ \ \ \ \ \ \ \ \
\ on }K\cup \cup _{\nu =1}^{m}\left[ 0,1\right] \times B_{\nu }, \\ 
=-\left( 1-\frac{2}{n}\right) \left( 1+\frac{1}{\varepsilon }\right) \delta
_{\nu }^{2}+2\lambda _{\mathbb{S}}>0\text{ on}\cup _{\nu =1}^{m}[2,\infty
)\times B_{\nu }.%
\end{array}%
\right.
\end{eqnarray*}%
On each $\left[ 1,2\right] \times B_{\nu }$, if $\gamma $ and $\delta
_{1},\cdots ,\delta _{m}$ are sufficiently small, then%
\begin{eqnarray}
&&\Delta \varphi -\left( 1-\frac{2}{n}\right) \left( 1+\frac{1}{\varepsilon }%
\right) \left\vert \nabla \varphi \right\vert ^{2}+2\lambda _{\mathbb{S}} 
\notag \\
&\geq &-C_{1}\left( \gamma +\left\vert h\right\vert +\left\vert \nabla
h\right\vert \right) +2\lambda _{\mathbb{S}}\geq \alpha ,  \label{l-est}
\end{eqnarray}%
where $C_{1}>0$ is some constant depending on $\phi $ and $\rho $. So we
have 
\begin{equation}
\Delta \varphi -\left( 1-\frac{2}{n}\right) \left( 1+\frac{1}{\varepsilon }%
\right) \left\vert \nabla \phi \right\vert ^{2}+2\lambda _{\mathbb{S}}\geq
\alpha _{1}\text{ on }M  \label{positive-b}
\end{equation}%
for some constant $\alpha _{1}>0$.

Therefore by Proposition \ref{Dirac-solve} and $\left( \ref{positive-b}%
\right) $, for any $f\in L_{\varphi }^{2}\left( M,\mathbb{S}\right) $, there
exists $u\in L_{\varphi }^{2}\left( M,\mathbb{S}\right) $ such that $Du=f$
and%
\begin{equation*}
\left\Vert u\right\Vert _{\varphi }^{2}\leq C_{2}\int_{M}\frac{\left\vert
f\right\vert ^{2}}{\Delta \varphi -\left( 1-\frac{2}{n}\right) \left( 1+%
\frac{1}{\varepsilon }\right) \left\vert \nabla \phi \right\vert
^{2}+2\lambda _{\mathbb{S}}}e^{-\varphi }\leq C\int_{M}\left\vert
f\right\vert ^{2}e^{-\varphi }
\end{equation*}%
for some constants $C_{2}$ and $C$. By elliptic regularity of $D$ and the
cylindrical structure on $M$, we have $\left\Vert u\right\Vert _{\varphi
}+\left\Vert \nabla u\right\Vert _{\varphi }\leq C\left\Vert f\right\Vert
_{\varphi }$, so $\left( \ref{weighted-sol}\right) $ is proved.

With the $L^{2}$-estimate $\left( \ref{weighted-sol}\right) $ ($\delta =0$
case) on the cylindrical manifold $M$, it is standard to derive the $L^{p}$
estimate $\left( \ref{weighted-Lp-cyl}\right) $ (c.f. Section 3.4 \cite{Don}%
). The theorem is proved.
\end{proof}


{\small Addresses:}

{\small Qingchun Ji}

{\small School of Mathematics, Fudan University, Shanghai 200433, China }

{\small Email: qingchunji@fudan.edu.cn}

\bigskip

{\small Ke Zhu}

{\small Department of Mathematics, Harvard University, Cambridge, MA 02138}

{\small Email: kzhu@math.harvard.edu}

{\small Current:}

{\small Department of Mathematics and Statistics }

{\small Minnesota State University Mankato }

{\small Mankato, MN 56001 }

{\small Email: ke.zhu@mnsu.edu}

\end{document}